\documentclass[a4paper,11pt]{amsart}

\usepackage{amsmath,amssymb,amsfonts,amsthm}

\usepackage{fancyhdr}
\usepackage{lmodern}
\usepackage[utf8]{inputenc}
\usepackage[T1]{fontenc}
\usepackage{fullpage}
\usepackage{enumerate}
\usepackage[all]{xy}
\usepackage{graphicx}
\usepackage[english]{babel}

\newdir{ >}{{}*!/-5pt/\dir{>}}

\newtheorem{theorem}{Theorem}[section]    
\newtheorem{definition}[theorem]{Definition}
\newtheorem{proposition}[theorem]{Proposition}
\newtheorem{lemma}[theorem]{Lemma}    
\newtheorem{corollary}[theorem]{Corollary}
\newtheorem{conjecture}[theorem]{Conjecture}
\newtheorem{remark}[theorem]{Remark}

\numberwithin{equation}{section}

\newcommand*{\defeq}{\mathrel{\vcenter{\baselineskip0.5ex \lineskiplimit0pt
                     \hbox{\scriptsize.}\hbox{\scriptsize.}}}%
                     =}

\newcommand{\apl}{A_{PL}}

\newcommand{\relm}{(A\otimes A \otimes \Lambda Z,D)}
\newcommand{\relmi}{A\otimes A \otimes \Lambda Z }
\newcommand{\ata}{A\otimes A}

\newcommand{\lvd}{\ahat}
\newcommand{\lvlv}{\ahat\otimes \ahat}
\newcommand{\lvvzd}{(\ahat \otimes \ahat \otimes \Lambda Z, D)}

\newcommand{\bq}{\mathbb{Q}}

\newcommand{\eqaa}{\simeq_{\ata}}

\newcommand{\ssn}{ss^{-n}}
\newcommand{\smn}{s^{-n}}

\newcommand{\bs}{\setminus}

\newcommand{\ahat}{\hat{A}}

\def\commutatif{\ar@{}[rd]|{\circlearrowleft}}
\def\pullback{\ar@{}[rd]|{pb}}
\def\comhotopie{\ar@{}[rd]|{\sim}}
\newcommand{\oid}{\operatorname{id}}

\newcommand{\qi}{\stackrel{\simeq}\longrightarrow}

\title{Rational model of the configuration space of two points in a simply connected closed manifold}
\author{Hector Cordova Bulens}

\begin{document}
%
\maketitle
\begin{abstract}
Let $M$ be a simply connected closed manifold of dimension $n$. We study the rational homotopy type of the configuration space of 2 points in $M$, $F(M,2)$. When $M$ is even dimensional, we prove that the rational homotopy type of $F(M,2)$ depends only on the rational homotopy type of $M$. When the dimension of $M$ is odd, for every $x\in H^{n-2} (M, \bq)$, we construct a commutative differential graded algebra $C(x)$. We prove that for some $x \in H^{n-2} (M; \bq)$, $C(x)$ encodes completely the rational homotopy type of $F(M,2)$. For some class of manifolds, we show that we can take $x=0$.

\end{abstract}

\section{Introduction}
Let $M$ be a simply connected closed manifold of dimension $n$. The configuration space of two points in $M$ is the space 
\[F(M,2)= \{(x,y) \in M\times M : x\neq y \}= M\times M \bs \Delta(M)\]
where $\Delta \colon M \hookrightarrow M\times M$ is the diagonal embedding. We have an obvious inclusion $F(M,2) \hookrightarrow M\times M$.  

Our goal in this paper is to study the rational homotopy type of $F(M,2)$. Recall that by the theory of Sullivan, the rational homotopy type of a simply connected space is encoded by a commutative differential graded algebra (CDGA for short), which is called a \emph{rational model} of the space. By Poincar\'e duality of the manifold, $M$ admits a Poincar\'e duality CDGA model $(A,d)$ (see Section \ref{Section 2}). There exists an element called the \emph{diagonal class} $\Delta \in (\ata)^n$ generalizing the classical diagonal class in $H^*(M; \bq) \otimes H^*(M; \bq)$.

 In \cite{LSRHT} it is shown that $\ata / (\Delta)$, where $(\Delta)$ is the ideal generated by $\Delta$ in $\ata$, is a CDGA model of $F(M,2)$ when $M$ is at least \emph{$2$-connected}. This results implies that the rational homotopy type of $F(M,2)$ depends only on the rational homotopy type of $M$ for a $2$-connected closed manifold. On the other side, Longoni and Salvatore \cite{LSFM2} have constructed an example of two connected (but not simply connected) closed manifolds that are homotopy equivalent but such that their configuration spaces of 2 points are not, even rationally. 

The goal of this paper is to discuss the rational homotopy type of $F(M,2)$ where $M$ is $1$-connected. When the dimension of $M$ is \emph{even} we show that (see Theorem \ref{T:CDGA_FM2_even}), as in the $2$-connected case, the CDGA $\ata / (\Delta)$ is a rational model of $F(M,2)$. Actually we show that $\ata \to \ata/ (\Delta)$ is a CDGA model of the inclusion $F(M,2) \hookrightarrow M\times M$. 

The \emph{odd} dimension case is more complicated. For any element $x \in H^{n-2}(M; \bq)$ we will construct a CDGA $C(x)$ and a map 
$\ata \to C(x).$
One of the main results of this paper (see Theorem \ref{T:CDGA_FM2_odd} and Corollary \ref{C:CDGA_FM2_odd}) is that for \emph{some} $x \in H^{n-2}(M; \bq)$, this map is a rational model of $F(M,2) \hookrightarrow M\times M$. 
Note that, for $x=0$, $C(0)$ is equivalent to the CDGA $\ata/ (\Delta)$; when $x\neq 0$, $C(x)$ is a multiplicatively twisted version of such a quotient (see Definition \ref{d:cxi}).  

In Section \ref{model_s2xs3} we introduce the notion of an \emph{untwisted manifold}. We will say that a manifold $M$ is \emph{untwisted} if $C(0)$ is a CDGA model of $F(M,2)$ (see Definition \ref{d:untwisted}).  In \cite{LSRHT} it is shown that all 2-connected closed manifolds are untwisted and in Section \ref{s:CDGA_FM2_even} we will show that every simply connected closed manifold of even dimension is untwisted. We prove in Section \ref{model_s2xs3} that the product of two untisted manifolds is untwisted. It is still an open question to know if all manifolds are untwisted. 

\subsection*{Acknowledgements} I would like to thank Yves Felix for suggesting that Proposition \ref{p:exemple} could be obtained from example $A=(H^*(S^2\times S^3; \bq),0)$.  I also thank my advisor Pascal Lambrechts for making this work possible through his advice and encouragement.

%
%

\section{Basic notions}\label{Section 2}
In this section we review the definition of a mapping cone of a module map and we describe a CGA structure on it. We also recall the definition and properties of a Poincar\'e Duality CDGA. 
In this paper we will use the usual tools of rational homotopy theory as developped for example in \cite{FHT}.
\subsection{Mapping cones}
\label{s:mc}
Let $R$ be a CDGA and let $A$ be an $R$-dgmodule. We will denote by $s^k A$ the k-th suspension of $A$ defined by $(s^k A)^n \cong A^{n+k}$ as a vector space and with an  $R$-dgmodule structure defined by $r\cdot (s^k a)= (-1)^{k|r|}s^k (r\cdot a)$ and $d(s^k a)= (-1)^k s^k (da)$ for $a\in A$ and $r\in R$.  
If $f \colon B \to A$ is an $R$-dgmodule morphism, \emph{the mapping cone} of $f$ is the $R$-dgmodule 
\[C(f)\defeq (A\oplus_f sB, \delta)\]
defined by $A\oplus sB$ as an $R$-dgmodule and $\delta (a,sb) = (d_A (a) +f(b), -sd_B(b))$. This $R$-dgmodule can be equipped with a commutative graded algebra (CGA) structure, that respects the $R$-dgmodule structure, characterised by the fact that $(s b) \cdot (sb')=0$, for $b,b' \in B$. Precisely, we have  the multiplication 
\[\mu : C(f)\otimes C(f) \to C(f) \ , \ c_1 \otimes c_2 \mapsto c_1 \cdot c_2, \]
such that, for homogeneous elements $a,a' \in A$ and $b,b' \in B$. 
\begin{itemize}
\item[$(i)$] $ \mu (a\otimes a')$ = $a\cdot a'$ ,\vspace{1mm}
\item[$(ii)$] $\mu (a \otimes sb')= (-1)^{|a|} s(a\cdot b') $ ,\vspace{1mm}
\item[$(iii)$] $\mu (sb \otimes a')= (-1)^{|b| |a'|} s(a'\cdot b) $ ,\vspace{1mm}
\item[$(iv)$] $\mu(sb \otimes sb')= 0$. 
\end{itemize}
 We will call this structure the \emph{semi-trivial structure} on the mapping cone.
\subsection{Poincar\'e duality CDGA}

A \emph{Poincar\'e duality} CDGA of formal dimension $n$ is a triple $(A,d, \epsilon)$ such that 
\begin{itemize}
\item $(A,d)$ is a CDGA;
\item $\epsilon \colon A^n \to \bq$ is a linear map such that $\epsilon (dA^{n-1})=0$ (one can think of $\epsilon$ as an orientation of the Poincar\'e duality CDGA $A$);
\item For each $k\in \mathbb{Z}$ 
\[\begin{array}{ccc}
A^k \otimes A^{n-k} &\to& \bq \\
a\otimes b & \mapsto & \epsilon (a\cdot b)
\end{array}\] is non degenerate, i.e., if $a\in A^k$ and $a\neq0$ then there exists $b\in A^{n-k}$ such that $\epsilon (a\cdot b)\neq 0$. 
\end{itemize}
 Let $\{a_i\}_{i=1}^N$ be an homogeneous basis of $A$. There exists a \emph{Poincar\'e dual} basis $\{a_j^*\}_{j=1}^N$ characterised by the fact that $\epsilon (a_i \cdot a^*_j)=\delta_{ij}$.

One of the main results concerning this algebras is the following one. 
\begin{theorem}{\cite[Theorem 1.1]{LSPDCDGA}}
Let $(A,d)$ be a CDGA such that $H^*(A,d)$ is a simply-connected Poincar\'e duality algebra of formal dimension $n$. Then there exists $(A',d')$ a Poincar\'e duality CDGA of formal dimension $n$  weakly equivalent to $(A,d)$.
\end{theorem}
\noindent
As a direct consequence of this result we have that  every simply connected closed manifold admits a Poincare duality CDGA model.

\section{A dgmodule model of $F(M,2)$}

The following result is an evident reformulation of the results of \cite[Theorem 10.1]{LSRDG} and \cite{LSRHT}. 

Let $(A,d, \epsilon)$ be a Poincar\'e duality CDGA of formal dimension $n$. Let $\{a_i\}_{i=1}^N$ be a homogeneous basis of $A$ and denote by $\{a_i^*\}_{i=1}^N$ its Poincar\'e dual basis. Denote by 
\[\Delta := \sum_{i=1}^{N} (-1)^{|a_i|} a_i\otimes a_i^* \in (\ata)^n\]
the diagonal class in $\ata$. There is an obvious $\ata$-module structure on $s^{-n}A$ defined, for homogeneous elements $a,x,y \in A$,  by $(x\otimes y) \cdot s^{-n} a = (-1)^{(n|x| + n |y| + |a||y|)} s^{-n} x\cdot a \cdot y$

The map 
\[\Delta^! \colon \smn A \to \ata\]
defined by $\Delta^! (\smn a) = \Delta \cdot (1\otimes a)$ is an $\ata$-dgmodule map (see \cite[Proposition 5.1]{LSRHT}). Therefore the mapping cone $C(\Delta^!)=\ata \oplus_{\Delta^!} ss^{-n}A$ is an $\ata$-dgmodule.

As a direct consequence of \cite[Theorem 10.1]{LSRDG} we obtain the following result
\begin{proposition}
\label{p:dgmodule_model}
Let $(A,d,\epsilon)$ be a Poincar\'e duality model of $M$ of formal dimension $n$. Let $\ahat$ be a CDGA such that we have 
\[\xymatrix{(A,d) &\ar[l]_-{\simeq} \ahat\ar[r]^-{\simeq} & \apl(M)}\]
a zig-zag of CDGA quasi-isomorphisms. The CGDAs $A\otimes A$, $\apl(M\times M)$ and $\apl(F(M,2))$ inherit an  $\ahat\otimes \ahat$-dgmodule structure. Denote by $\Delta$ the diagonal class in $\ata$ and let  $\Delta^! \colon \smn A \to \ata$ be the morphism defined by $\Delta^! (\smn a) =\Delta(1\otimes a)$ for all $a\in A$. Then the mapping cone 
 \[C(\Delta^!)= \ata \oplus_{\Delta^!} ss^{-n}A\] 
is an $\ahat \otimes \ahat$-dgmodule weakly equivalent to $\apl(F(M,2))$. 
\end{proposition}

We will often use the following result.  
\begin{lemma}{ \cite[Lemma 5.2]{LSRHT}}
Let $(A,d,\epsilon)$ be a Poincar\'e duality CDGA. The mapping cone $C(\Delta^!)=\ata \oplus_{\Delta^!} \ssn A$ equipped with the semi-trivial structure is a CDGA .
\end{lemma}

This suggests that $C(\Delta^!)$ equipped with the semi-trivial structure (see Section \ref{s:mc}) is a natural candidate to be a CDGA model of $F(M,2)$. The main result of \cite{LSRHT} proves that it is when the manifold is 2-connected. In the following section we show that it also is when the manifold is 1-connected and of even dimension.

\section{Rational model of $F(M,2)$ for a simply connected manifold of even dimension}
\label{s:CDGA_FM2_even}

Let $M$ be  a simply connected closed manifold of even dimension $n$. We will prove the following result:

\begin{theorem}
\label{T:CDGA_FM2_even}
Let $M$ be a simply connected closed manifold of even dimension $n$. Let $A$ be a 1-connected Poincar\'e duality model of $M$ of formal dimension $n$ and denote by $\Delta$ the diagonal class in $\ata$. Let $\Delta^! \colon \smn A \to \ata$ be the morphism such that $\Delta^!(\smn a)= \Delta (1\otimes a)$ for $a\in A$. The mapping cone of $\Delta^!$
\[C(\Delta^!)= \ata\oplus_{\Delta^!} \ssn A\] 
 equipped with the semi-trivial structure is a CDGA and the map \[\oid \oplus 0 \colon \ata \to C(\Delta^!)\] defined by $\oid \oplus 0 (a\otimes b)= (a\otimes b , 0)$ is a CDGA  model of the inclusion $F(M,2)\hookrightarrow M\times M$.
\end{theorem}

A direct corollary of the theorem is the rational homotopy invariance of $F(M,2)$. Specifically:
\begin{corollary}
Let $M$ be a simply connected closed manifold of even dimension. The rational homotopy type of $F(M,2)$ is completely determined by the rational homotopy type of $M$.
\end{corollary}
\begin{proof}
Let $n= \dim M$. If $n=2$ then $M\cong S^2$ and the result is obvious. If $n\geq 4$, a transversality argument shows that, if $M$ is simply connected, the configuration space $F(M,2)$ is simply connected. Hence the result is a direct consequence of Theorem \ref{T:CDGA_FM2_even}.
\end{proof}

                                                                                                                                                                                                                                                                                                                                                                                                                                                                                                                                                                                                                                                                                                                                                                                                                                                                                                                                                                                                                                                                                                                                                                                                                                                                                                                                                                                                                                                                                                                                                                                                                                                                                                                                                                                                                                                                                                                                                                                                                                                                                                                                                                                                                                                                                                                                                                                                                                                                                                                                                                                                                                                                                                                                                                                                                                                                                                                                                                                                                                                                                                                                                                                                                                                                                                                                                                                                                                                                                                                                                                                                                                                                                                                                                                                                                                                                                                                                                                                                                                                                                                                                                                                                                                                              The rest of this section is devoted to the proof of Theorem \ref{T:CDGA_FM2_even}.                                                                                                                                                                                                                                                                                                                                                                                                                                                                                                                                                                                                                                                                                                                                                                                                                                                                                                                                                                                                                                                                                                                                                                                                                                                                                                                                                                                                                                                                                                                                                                                                                                                                                                                                                                                                                                                                                                                                                                                                                                                                                                                                                                                                                                                                                                                                                                                                                                                                                                                                                                                                                                                                                                                                                                                                                                                                                                                                                                                                                                                                                                                                                                                                                                                                                                                                                                                                                                                                                                                                                                                                                                                                                                                                                                                                                                                                                                                                                                                                                                                                                                                                                                                  
 For this we, use the following obvious result.                                                                                                                                                                                                                                                                                                                                                                                                                                                                                                                                                                                                                                                                                                                                                                                                                                                                                                                                                                                                                                                                                                                                                                                                                                                                                                                                                                                                                                                                                                                                                                                                                                                                                                                                                                                                                                                                                                                                                                                                                                                                                                                                                                                                                                                                                                                                                                                                                                                                                                                                                                                                                                                                                                                                                                                                                                                                                                                                                                                                                                                                                                                                                                                                                                                                                                                                                                                                                                                                                                                                                                                                                                                                                                                                                                                                                                                                                                                                                                                                                                                                                                                                                                                                                                                                                                                                                                                                                                                                                                                                                                                                                                                                                                                                                                                                                                                                                                                                                                                                                                                                                                                                                                                                                                                                                                                                                                                                                                                                                                                                                                                                                                                                                                                                                                                                                                                                                                                                                                                                                                                                                                                                                                                                                                                                                                
%

\begin{lemma} 
\label{module_adgc_pour_fm2}
Let $\rho \colon R \to Q$ be a CDGA morphism and $R\otimes \Lambda Z$ be a relative Sullivan algebra.  
Let $\varphi: R\otimes \Lambda Z \to Q$ be an $R$-dgmodules morphism such that $\varphi|_R= \rho$. The morphism $\varphi$ is a CDGA morphism if for all $z,z' \in \Lambda^+ Z$, $\varphi(z\cdot z')=\varphi(z)\cdot \varphi(z')$.
\end{lemma}

Let us prove Theorem \ref{T:CDGA_FM2_even}.

\begin{proof}[Proof of theorem \ref{T:CDGA_FM2_even}]
Let $A$ be a 1-connected (i.e. $A^0=\bq$, $A^1=0$) Poincar\'e Duality model of $M$ of formal dimension $n$. 
Take $\ahat$ a CDGA such that we have a zig-zag
\[\xymatrix{(A,d) &\ar[l]_-{\simeq} \ahat \ar[r]^-{\simeq} & \apl(M)}\]
of CDGA quasi-isomorphisms. 

The CDGAs  $A$, $A\otimes A$, $\apl(M\times M)$ and $\apl(F(M,2))$ inherit an $\ahat\otimes \ahat$-dgmodule structure. Taking a relative Sullivan model of  $f\colon \ahat\otimes \ahat \to \apl(M\times M) \to~\apl(F(M,2)),$ we get the following commutative diagram
\[\xymatrix{\lvd \otimes \lvd \ar[r]\ar@{ >->}[dr] &\apl(M\times M) \ar[r] & \apl(F(M,2))\\
&(\ahat \otimes \ahat \otimes \Lambda Z, D). \ar[ur]_m^{\simeq}}\]
Since $H^{<n}(M\times M, F(M,2))=0$ and $H^{n} (M\times M, F(M,2) ; \bq)\cong \bq$ the vector space $Z$ can be taken such that $Z^{<n-1}=0$ and $Z^{n-1}\cong \bq$. Let $u$ denote a generator of $Z^{n-1}$.  By Proposition 
\ref{p:dgmodule_model}, $\apl(F(M,2))$ and $C(\Delta^!)$ are weakly equivalent as $\ahat\otimes \ahat$-dgmodules. Since $\lvvzd$ is a cofibrant $\ahat\otimes \ahat$-dgmodule we have a direct quasi-isomorphism of $\ahat\otimes \ahat$-dgmodules 
\[\theta\colon  \lvvzd \stackrel{\simeq}\longrightarrow C(\Delta^!).\]

Denote by $\omega \in A^n$ a generator of the vector space $A^n\cong \bq$. The differential ideal $I= \langle \omega \otimes \omega, ss^{-n}\omega\rangle \subset C(\Delta^!)$ is acyclic and the projection 
$\pi  \colon C(\Delta^!) \to C(\Delta^!) / I$
is a CDGA quasi-isomorphism. Therefore the composition 
\[\varphi\defeq \pi \circ \theta \colon \lvvzd \to C(\Delta^!)/I\]
is a quasi-isomorphism of $\lvlv$-dgmodules. To show that it is a quasi-isomorphism of CDGA it suffices to show, by Lemma \ref{module_adgc_pour_fm2}, that for all $z,z' \in \Lambda^+ Z$:

\begin{equation}
\label{alg}
\varphi(z\cdot z')=  \varphi (z)\varphi( z').
\end{equation}
Recall that $Z^{<n-1}=0$,
 \begin{itemize}
\item If $|z|> n-1$ or $|z'|>n-1$ then the degree of expression \eqref{alg}  is $\geq 2n-1$, hence both terms are zero since $(C(\Delta^!)/I)^{\geq 2n-1}=~0$.
\item If $|z|=n-1$ and $|z'|=n-1$, by linearity and since $Z^{n-1}=\bq\cdot u$ it suffices to look at the case $z=z'=u$. Since $n-1$ is odd, by graded commutativity, $u^2=0$ and $\varphi(u^2)=\varphi(0)=\varphi(u)\cdot\varphi(u)=0.$  
\end{itemize}
So, we have the following zig-zag of CDGA quasi-isomorphisms, 
\[\xymatrix{\apl(F(M,2))&\ar[l]^{\simeq}_{m}\lvvzd \ar[r]_-{\simeq}^-{\pi \circ \theta} & C(\Delta^!) /I & \ar[l]^-{\simeq}_-{\pi} C(\Delta^!),}\]
which proves that $C(\Delta^!)$ is a rational model of $F(M,2)$.

Multiplying by a non zero rational number we can suppose that $\theta(1)=1$ and since $\theta$  is an $\lvlv$-dgmodules map we have that for $a, b \in \ahat$ 
\[\theta (a\otimes b) = (a\otimes b)\cdot \theta(1)= \rho(a\otimes b).1=\left((\oid\oplus 0) \circ \rho\right)(a\otimes b).\]
In other words, the following diagram commutes 
\[\xymatrix{\ahat\otimes \ahat\ar[d] \ar[r]^{\rho} & \ata \ar[d]^{\oid\oplus 0} \\
\lvvzd \ar[r]^-{\theta} & C(\Delta^!),}\]
and proves that the morphism $\oid\oplus 0 \colon \ata \to C(\Delta^!)$ is a CDGA model morphism of \linebreak$\apl(M\times M) \to \apl(F(M,2))$.
\end{proof}

In \cite[Lemma 5.4 and Lemma 5.5]{LSRHT} it is shown that, if $(\Delta)$ denotes the ideal generated by $\Delta \in \ata$, $\frac{\ata}{(\Delta)}$ is a CDGA quasi-isomorphic to the CDGA $C(\Delta^!)$ endowed with the semi-trivial structure. 

\begin{corollary}
\label{corollary_ata}
Let $M$ be a 1-connected closed manifold of even dimension. Let $A$ be a $1$-connected Poincar\'e duality CDGA model of $M$. Then \[\ata \to \frac{\ata}{(\Delta)}\] is a CDGA model of $F(M,2)\hookrightarrow M\times M$.
\end{corollary}

\section{Rational model of $F(M,2)$ for a simply connected manifold of odd dimension}
\label{s:odd}
Let $M$ be a closed, simply connected manifold of odd dimension $n$. Let $A$ be a Poincar\'e duality model of $M$. The construction of a rational model for $F(M,2)$ in this case is more complicated. To construct such a model we will introduce dgmodules weakly equivalent to $C(\Delta^!)$  that admit multiplicative structures other than the semi-trivial structure used above. Explicitly, for every $\xi \in (\ata)^{2n-2}$ we will construct a CDGA $C(\xi)$ where the multiplicative structure depends on $\xi$ (see Definition \ref{d:cxi}). We will show that one of these CDGAs $C(\xi)$ is the rational model of the configuration space $F(M,2)$. Later on we will discuss the fact that the choice of such a multiplicative structure is determined by the choice of an element in $H^{n-2}(M; \bq)$. 

\subsection{The algebra $C(\xi)$}

Let $A$ be a 1-connected Poincar\'e duality CDGA of formal dimension $n$ odd. Denote by $\Delta$ the diagonal class in $\ata$. 
 Recall that the morphism \[\Delta^! \colon \smn A \to \ata \ ; \ \smn a \mapsto \Delta (1\otimes a)\] is an  $\ata$-dgmodule morphism. Hence, the mapping cone \[C(\Delta^!)=\ata \oplus_{\Delta^!} \ssn A\] is an $\ata$-dgmodule.  

We will show subsequently that a quotient of the mapping cone $C(\Delta^!)$ by an acyclic $\ata$-subdgmodule admits other CDGA structures than the semi-trivial one.
 It is easy to see that the $\ata$-subdgmodule $(C(\Delta^!))^{\geq 2n-1}=\langle \omega \otimes \omega, \ssn \omega \rangle\subset C(\Delta^!) $ is acyclic. Hence, the quotient  \[C:= C(\Delta^!)/C(\Delta^!)^{\geq 2n-1}\] is an $\ata$-dgmodule and the canonical projection $\pi: C(\Delta^!)\to C$ is a quasi-isomorphism of $\ata$-dgmodules.

Let us define a multiplication over $C$ such that it becomes a CDGA. We want this multiplication to be compatible with the $\ata$-dgmodule structure, namely that the composition
\[\xymatrix{\ata \ar@{^{(}->}[r]& C(\Delta^!)\ar[r]^-\pi &C}\]
is a CDGA morphism. 

Since $C$ is the quotient of  $\ata \oplus_{\Delta^!} ss^{-n}A$ by the vector space $\langle \omega\otimes \omega, ss^{-n} \omega\rangle$, we have an obvious isomorphism
\[C\cong (\ata)^{<2n} \oplus_{\Delta^!} ss^{-n}(A^{<n}).\]
To define a multiplication over $C$ :
\[\mu : C\otimes C \to C \ , \ c_1 \otimes c_2 \mapsto c_1 \cdot c_2, \]
it suffices to define the products
\begin{itemize}
\item[$(i)$] $(a\otimes b) \cdot (a'\otimes b'),$ \vspace{1mm}
\item[$(ii)$] $(a\otimes b) \cdot ss^{-n}x'$ and $ss^{-n} x \cdot (a'\otimes b'),$\vspace{1mm}
\item[$(iii)$] $(ss^{-n}x)\cdot (ss^{-n}x'),$
\end{itemize}
\vspace{1mm}
for $a\otimes b$, $a'\otimes b' \in (\ata)^{<2n}$ and $x,x' \in A^{<n}$ homogeneous elements. \vspace{1mm}

The products $(i)$ and $(ii)$ are completely determined by the compatibility with the $\ata$-module structure and the graded commutativity. Specifically :
\begin{itemize}
\item[$(i)$] $(a\otimes b)\cdot (a' \otimes b')=(-1)^{|a'||b|} a\cdot a' \otimes b\cdot b'$ 
\vspace{1mm}
\item[$(ii)$] $(a\otimes b)\cdot ss^{-n} x'=(-1)^{(n-1)(|a|+|b|)+|b||x|}ss^{-n}a\cdot x' \cdot b$ \vspace{1mm}
\\ and $(ss^{-n}x)(a'\otimes b')=(-1)^{|a'||x|}ss^{-n}a'\cdot x \cdot b'.$
\end{itemize}
\vspace{1mm}
For $(iii)$, note that, if $|x| + |x'| \geq 1$, then for degree reasons 
\vspace{1mm}
\begin{itemize}
\item[$(iii)$] $(ss^{-n}x)\cdot  (ss^{-n}x') =0$.
\end{itemize}
\vspace{1mm}
Hence, the only still undefined product  is \vspace{1mm}
\begin{itemize}
\item[$(iv)$] $(ss^{-n}1)\cdot (ss^{-n}1) \in C^{2n-2}=(\ata)^{2n-2}$.
\end{itemize}
\medskip

The following result shows that if $n$ is odd, the product $ss^{-n}1 \cdot ss^{-n}1$ can be arbitrarily chosen in $(\ata)^{2n-2}$.
\begin{proposition}
Let $\xi \in (\ata)^{2n-2}$. If $n$ is odd there exists a CDGA structure on $C$ compatible with the $\ata$-dgmodule structure such that  \[ss^{-n}1\cdot ss^{-n}1=\xi.\]
\end{proposition}

%
%
\begin{proof} 
Endow $C$ with the multiplicative structure described above setting that $(\ssn 1)^2=\xi.$  Let us show that this multiplication endows $C$ with a CDGA structure.
\begin{itemize}
\item Associativity:
For $x,y,z\in A$, for degree reasons
\[ (ss^{-n} x\cdot ss^{-n} y)\cdot ss^{-n}z=0= ss^{-n} x\cdot ( ss^{-n} y\cdot ss^{-n}z).\] 
For $x, y \in A$ and $a\otimes b\in \ata$ such that $|a\otimes b|> 0$ then $|(a\otimes b)| + |ss^{-n}x| + |ss^{-n} y|\geq 2n-1$. Therefore, for degree reasons \[(a\otimes b)\cdot (ss^{-n}x\cdot ss^{-n}y)=(a\otimes b \cdot ss^{-n}x)\cdot ss^{-n}y=0.\]
In the other cases the associativity is granted by the $\ata$-dgmodule structure of $C$.

\item Commutativity: The multiplicative structure is defined in such a way that the multiplication is commutative.
\item The Leibniz rule for the differential:
Denote by $\bar{\delta}$ the differential in $C$.
Let us show that for $c$, $c'$ homogeneous elements in $C$,  
\begin{equation}
\label{Leib}
\bar\delta( c\cdot c')-  (\bar\delta (c)\cdot c')-(-1)^{|c|}(c\cdot \bar\delta (c'))=0.
\end{equation}
\begin{itemize}
\item If $c, c' \in \ata$ then (\ref{Leib}) is true since $\ata$ is a CDGA and $\bar{\delta}|_{\ata}=d|_{\ata}$. 
\item If $c\in \ata$ and $c' \in ss^{-n} A$ then the left term of equation (\ref{Leib}) is zero since $\bar{\delta}$ is a differential of $\ata$-dgmodules. 
\item If $c\in ss^{-n} A$ and $c'\in \ata$, (\ref{Leib}) is established by the commutativity of the multiplication. 
\item If $c, c' \in ss^{-n}A$, then the degree of the expression is at least $2n-2+1=2n-1$. Therefore, for degree reasons, the left term of the equation is zero. \nolinebreak 
\end{itemize}
\end{itemize}
 \end{proof}

\begin{definition}
\label{d:cxi}
Let $\xi \in (\ata)^{2n-2}$. The CDGA $C(\xi)$ is the $\ata$-dgmodule  $C(\Delta^!)/C(\Delta^!)^{\geq 2n-1}$ equipped with the only structure of an algebra compatible with the $\ata$-dgmodule structure and such that $\ssn 1 \cdot \ssn 1=\xi$.  We will denote by  
\[\overline{\oid} \oplus 0 \colon \ata \to C(\xi)\] the composition of the inclusion $\oid\oplus 0: \ata \to \ata \oplus_{\Delta^!} \ssn A$ and the canonical projection $C(\Delta^!) \to C(\xi)$. 
\end{definition}

\subsection{CDGA model of $F(M,2)$ for a simply connected manifold of odd dimension}

Let $M$ be a simply connected closed manifold of odd dimension $n$. Let $A$ be a $1$-connected Poincar\'e duality model of $M$. In the previous sections we have showed that for all $\xi \in (\ata)^{2n-2}$ we can construct a CDGA $C(\xi)$ characterised by $(\ssn 1)^2= \xi$. The objective of this section is to show that there exists $\xi \in (\ata)^{2n-2}$ such that $C(\xi)$ is a rational model of $F(M,2)$. Specifically we will prove the following:
\begin{theorem}
\label{T:CDGA_FM2_odd}

Let $M$ be a simply connected closed manifold of odd dimension $n$ and let $A$ be a $1$-connected Poincar\'e duality model of $M$ of formal dimension $n$. There exists $\xi \in (\ata)^{2n-2}$, such that the map 
\[\overline{\oid}\oplus 0\colon \ata \to C(\xi)\]
is a CDGA model of the inclusion $F(M,2) \hookrightarrow M\times M$. 
\end{theorem}
\begin{remark} The model $\ata$ of $M\times M$ that is mentioned in the theorem is the one naturaly induced by the CDGA model $A$ of $M$, i.e. the zig-zag of CDGA quasi-isomorphism linking $A\otimes A$ and $\apl (M\times M)$ is naturaly induced by the zig-zag of quasi-isomorphisms between $A$ and $\apl (M)$. We will use this fact in Section \ref{Section 7} to introduce Definition \ref{d:cxi_cxip}.
\end{remark}
\medskip

\begin{remark} We will see in Corollay \ref{C:CDGA_FM2_odd} that actually the model $C(\xi)$ is determined by a cohomology class $x \in H^{n-2}(M; \bq)$.
\end{remark}
\begin{proof}[Proof of Theorem \ref{T:CDGA_FM2_odd}]

Take $\ahat$ a CDGA such that we have
\[\xymatrix{A &\ar@{->>}[l]_-{\simeq} \ahat\ar@{->>}[r]^-{\simeq} & \apl(M)}\]
a zig-zag of surjective CDGA quasi-isomorphisms. This can be done easily using standard techniques of model categories.  The CDGAs $A$, $\ata$, $\apl(M\times M)$ and $\apl(F(M,2))$ inherit a $\ahat\otimes \ahat$-dgmodule structure.

Taking a relative Sullivan model of the map 
$f \colon \ahat\otimes \ahat \qi \apl(M\times M) \to \apl(F(M,2))$
we obtain the following commutative diagram
\[\xymatrix{\lvd \otimes \lvd \ar[r]^-{\simeq}\ar[dr] &\apl(M\times M) \ar[r] & \apl(F(M,2))\\
&(\ahat \otimes \ahat \otimes \Lambda Z, D). \ar[ur]_m^{\simeq}}\]
We can take $Z$ such that $Z^{<n-1}=0$ and $Z^{n-1}= u\cdot \bq$. By construction, $Du= \tilde \Delta$ where $\tilde \Delta \in \ahat\otimes \ahat$ is such that $[\tilde \Delta]$ is a generator of $\ker H^* (f)$. The surjectivity of $\rho \colon \lvlv\stackrel{\simeq}{\twoheadrightarrow} \ata$ allows us to take $\tilde \Delta$ such that $\rho(\tilde{\Delta})= \Delta.$ Indeed, let $\hat{\Delta} \in \ata \cap \ker d$ a generator of $\ker H^* (f)$. Since, $\hat \Delta$ represents the diagonal class in $H^*(\ahat \otimes \ahat) \cong H^*(M\times M)$, 
\[H^*(\rho)([\hat \Delta])= [\Delta] \in H^*(\ata).\]
Hence, $\rho (\hat \Delta) = \Delta + d \epsilon$, where $\epsilon \in (\ata)^{n-1}$. Since $\rho$ is surjective, there exists $\gamma \in \ahat \otimes \ahat$ such that $\rho(\gamma)= \epsilon$. Set $\tilde \Delta = (\hat{\Delta} -d\gamma)$. We have that $d\tilde\Delta = 0$ and $\rho(\tilde \Delta)= \Delta +d\epsilon -d\epsilon= \Delta$. 

By Proposition \ref{p:dgmodule_model}, $C(\Delta^!)= \ata \oplus_{\Delta^!} \ssn A$ and $\apl(F(M,2))$ are weakly equivalent as $\lvlv$-dgmodules. As $\lvvzd$ is a cofibrant $\lvlv$-dgmodule we have an $\lvlv$-dgmodules quasi-isomorphism 
\[\theta \colon  \lvvzd \qi C(\Delta^!).\] 
Multiplying $\theta$ by a non zero rational number we can assume that $\theta(1)=1$ and as in the end of proof of theorem \ref{T:CDGA_FM2_even} we show that the following diagram commutes 
\begin{equation}
\label{diag1}
\xymatrix{\ahat\otimes \ahat\ar[d] \ar[r]^{\rho} & \ata \ar[d]^{\oid\oplus 0} \\
\lvvzd \ar[r]^-{\theta} & C(\Delta^!)}
\end{equation}
Let us show that $\theta(u)= \ssn 1 + \tau$ with $\tau \in (\ata)^{n-1} \cap \ker d$.
Indeed, we remark that for degree reasons 
\[\theta (u) = q \ssn 1 + \tau\]
with $q \in \bq$ and $\tau \in (\ata)^{n-1}$. Let us denote by $\delta$ the differential in $C(\Delta^!)$, since $\theta$ commutes with the differentials we have
\begin{eqnarray*}
\delta (\theta (u)) &= &\theta (Du), \\
\delta(q \ssn 1 + \tau) &=& \theta(\tilde \Delta),\\
q\Delta + d\tau &=& \rho(\tilde \Delta)=\Delta.
\end{eqnarray*}
Hence, 
\begin{equation}
\label{tau_delta}
d\tau = (1-q)\Delta.
\end{equation}
Since $[\Delta]\neq 0$ in $H^*(\ata; \bq)$, the equation \eqref{tau_delta} is satisfied if and only if $q=1$ and thus $\theta (u) = \ssn 1 +\tau$ with $d\tau=0$.
\medskip

Let $\xi = \theta(u^2)-\tau^2$ so $|\xi|=2n-2$. Since $A$ is 1-connected $(ss^{-n} A)^{2n-2}\cong A^{n-1}=~0$, thus $\xi \in C(\Delta^!)$ can be identified wiht an element of $(\ata)^{2n-2}$.  Take $C(\xi)$ the CDGA of Definition \ref{d:cxi}. The composition
\[\xymatrix{\varphi \colon \lvvzd \ar[r]^-{\theta}& C(\Delta^!) \ar[r]^{\pi}& C(\xi)}\]
is an $\lvlv$-dgmodules quasi-isomorphism. Let us show that $\varphi$ is an algebra morphism. By Lemma \ref{module_adgc_pour_fm2}, it suffices to show that  for $z$, $z' \in \Lambda^+ Z$
\begin{equation}
\label{alg1}
\varphi(z \cdot  z')=  \varphi(z)\varphi(z').
\end{equation}
We study the following cases 
 \begin{itemize}
\item If $|z|> n-1$ or $|z'|>n-1$  then the degree of the expression is $\geq 2n-1$, and so both terms of the equation are zero.
\item If $|z|=n-1$ and $|z'|=n-1$, by linearity we can suppose that $z=z'=u$. Note that, since $A$ is a 1-connected Poincar\'e duality CDGA, $\tau\cdot  ss^{-n} 1 \in s\smn (A^{n-1})=0$. We compute that 
\[(\varphi(u))^2=(s\smn 1 + \tau)^2=(s\smn1)^2 + \tau^2=\xi + \tau^2= \pi (\theta(u^2))\]
and  \[\varphi(u^2)=\pi(\theta(u^2)).\]
\end{itemize}
This proves that $\varphi$ is an algebra morphism. 

Since $\varphi$ is an $\lvlv$-dgmodules quasi-isomorphism and is a CDGA map, it is a CDGA quasi-isomorphism. The following zig-zag of CDGA quasi-isomorphisms
\[\xymatrix{\apl(F(M,2)) & \ar[l]^-{m}_-{\simeq} \lvvzd \ar[r]_-{\varphi}^-{\simeq} &C(\xi),}\]
proves that $C(\xi)$ is a CDGA model of $F(M,2)$.  Since the diagram \eqref{diag1} commutes, $\ata \to C(\xi)$ is a CDGA model morphism of $\apl(M\times M) \to \apl(F(M,2))$.
\end{proof}

\section{Untwisted manifolds}
\label{model_s2xs3}

 It is well known that for $M$ a simply connected closed manifold \[H^*(F(M,2))\cong \frac{H^*(M)\otimes H^*(M)}{([\Delta])},\] where $[\Delta]$ is the diagonal class in $H^*(M)\otimes H^*(M)$. It is easy to see that for $\xi = 0$ the CDGA $C(\xi)$ (see Definition \ref{d:cxi}) is quasi-isomorphic to $\frac{\ata}{(\Delta)}$ as a CDGA. This suggests "naively" that $C(0)$ is the good candidate to be the rational model of $F(M,2)$ and therefore it suggests the following definition:

\begin{definition}
\label{d:untwisted}
A closed manifold $M$ is \emph{untwisted} if  $\ata \to \ata/(\Delta)$ is a CDGA model of $F(M,2)\hookrightarrow M\times M$, where $A$ is a $1$-connected Poincar\'e Duality CDGA model of $M$.
\end{definition}

In \cite{LSRHT} it is shown that all 2-connected closed manifolds are untwisted and in Section \ref{s:CDGA_FM2_even} we showed that every simply connected closed manifold of even dimension is untwisted. We also have the following result:

\begin{proposition}
\label{p:untwisted}
The product of two untwisted manifolds is untwisted.
\end{proposition} 
\begin{proof}
Notice that if $X$ and $Y$ are topological spaces, the configuration space $F(X\times Y, 2)$ can be viewed as the pushout of the diagram 
\begin{equation}
\label{d:fxy}
\xymatrix{X^2 \times F(Y,2) &\ar@{_{(}->}[l] F(X, 2) \times F(Y,2) \ar@{^{(}->}[r]&F(X,2) \times Y^2}.
\end{equation}
Replacing $\Delta (X)\subset X\times X$ by an open tubular neighborhood in $F(X,2)= X\times X \bs \Delta(X)$ and similarly replacing $\Delta(Y)\subset Y\times Y$  by an open tubular neighborhood in $F(Y,2)$ we can look at pushout \eqref{d:fxy} as a homotopy pushout.

Suppose that $X$ and $Y$ are untwisted manifolds, let $C$ be a Poincar\'e duality CDGA model of $X$ and $B$ a Poincar\'e duality CDGA model of $Y$. By definition of untwisted manifolds $\frac{C\otimes C}{(\Delta_C)}$ is  a CDGA model of $F(X,2)$ and $\frac{B\otimes B}{(\Delta_B)}$ is a CDGA model of $F(Y,2)$. Therefore a CDGA model of diagram \eqref{d:fxy} is given by
\begin{equation}
\label{fxy2}
\xymatrix{B\otimes B \otimes \frac{C\otimes C}{(\Delta_C)} \ar@{->>}[r]& \frac{B\otimes B}{(\Delta_B)} \otimes \frac{C\otimes C}{(\Delta_C)} &\ar@{->>}[l]\frac{B\otimes B}{(\Delta_B)} \otimes C\otimes C.}
\end{equation}
As both maps are surjective the homotopy pullback of diagram \eqref{fxy2} is the actual pullback of the diagram and hence this pullback is a CDGA model of $F(X\times Y, 2)$. 

Let us denote by $PB$ this pullback. The universal property guarantees the existence of a unique map
\[\alpha \colon C\otimes C  \otimes B \otimes B \to PB.\]
It is easy to show that $\alpha$ is surjective and that $\ker \alpha = ( \Delta_C \otimes  \Delta_B)$. Thus, \[PB \cong C\otimes C \otimes B\otimes B /(\Delta_C \otimes \Delta_B).\]

The CDGA $A= C\otimes B$ is a CDGA model of $X\times Y$ and denote by $\Delta_A$ the diagonal class in $A$. There is a natural isomorphim $C\otimes C \otimes B  \otimes B\cong  A\otimes A $ which, by an adequate choice of basis, sends $\Delta_C \otimes \Delta_B$ to $\Delta_A$. So $A\otimes A / (\Delta_A)$ is a CDGA model of $F(X\times Y,2)$.  
\end{proof}


As a consequence we have that, if $N$ is a 2-connected odd dimensional closed manifold then $N\times \underbrace{S^2 \times \cdots \times S^2}_{\text{$k$ times}}$ is untwisted.
We conjecture the following:

\begin{conjecture}
Every simply connected closed manifold is untwisted.
\end{conjecture}

\section{Models $\ata \to C(\xi)$ depend only on a class $x\in H^{n-2}(M;\bq)$ }
\label{Section 7}
A natural question is whether two different $\xi, \xi' \in (\ata)^{2n-2}$ induce isomorphic CDGAs $C(\xi) \cong C(\xi')$ (at the end of this section we show that this is not always the case). More precisely, in view of Theorem 9 where $C(\xi)$ is seen as a CDGA under $\ata$, the right question is whether $C(\xi)$ and $C(\xi')$ are equivalent in the following sense:



\begin{definition}
\label{d:cxi_cxip}
Let $A$ be a Poincar\'e duality CDGA of formal dimension $n$ odd and $\xi, \xi' \in (\ata)^{2n-2}$. We say that the CDGAs $C(\xi)$ and $C(\xi')$ are \emph{weakly equivalent under $\ata$} and we write $C(\xi)\simeq_{\ata} C(\xi')$, if there exists a commutative CDGA diagram   
\[\xymatrix{&C(\xi) \\
\ata \ar[ur]^{\bar\oid\oplus 0}\ar@{ >->}[r] \ar[dr]_{\bar\oid\oplus 0}  & \relm \ar[u]_{\simeq} \ar[d]^{\simeq}\\
&C(\xi'),}\]
where $(\relmi, D)$ is a relative Sullivan algebra. \end{definition}


We have the following result:
\begin{proposition}
\label{p:xi_xip}
Let $A$ be a 1-connected Poincar\'e duality CDGA of formal dimension $n$ odd, and $\Delta$ be the diagonal class in $\ata$. If $\xi, \xi'  \in (\ata)^{2n-2}$ are such that $[\xi]=[\xi'] \text{ in } \frac{H^{2n-2}(\ata)}{([\Delta])}$, 
then $C(\xi) \simeq_{\ata} C(\xi')$.
\end{proposition}
\begin{proof}

We have a canonical $\ata$-dgmodule isomorphism 
$C(\xi) \cong C(\xi').$
Abusing of notation, we can look at this as an identification since these two CDGA differ simply by the value of the square $(s\smn 1)^2$.

As $\ata$-dgmodules we have
\[C(\xi) = ((\ata)^{<2n-1} \oplus ss^{-n} (A^{<n-1}) , \delta)\]
with $\delta (ss^{-n}a)= \delta(a\cdot \ssn 1)= da \cdot \ssn 1 + (-1)^{|a|}(a\otimes 1)\cdot \Delta$. 
Set
\[(\Delta)^{>n} =\{(a\otimes b)\cdot \Delta \ | \   a\otimes b \in (\ata)^+\}.\]
Since the map $a\mapsto (a\otimes 1)\cdot \Delta$ is injective (see \cite[Lemma 5.4]{LSRHT}), it is clear that
\[(\Delta)^{>n} \oplus \ssn 1 \cdot A^+\]
is an acyclic differential ideal in $C(\Delta^!)$ and then, the natural projection of this ideal is also a differential ideal in $C(\xi)$ and $C(\xi')$. Consider $S$ a supplementary of cocycles in $(\ata)^{2n-3}$. The ideal
\[I=S+d(S) + (\Delta)^{>n}+ss^{-n}\cdot A^+,\]
is an acyclic differential ideal in $C(\Delta^!)$ and so the natural projection of this ideal is also an acyclic ideal in $C(\xi)$ and $C(\xi')$. We have two CDGA quasi-isomorphisms 
 $C(\xi)\qi \nolinebreak C(\xi)/I$ and $C(\xi') \qi C(\xi')/I$. Since $\xi-\xi' \in (\Delta) + d\left((\ata)^{2n-3}\right),$ $\xi=\xi' \bmod I.$ Therefore, the identification of $\ata$-dgmodules $C(\xi)=C(\xi')$ induces a CDGA isomorphism $C(\xi)/I \cong C(\xi')/I.$
We have then the following CDGA diagram 
\[\xymatrix{ & C(\xi) \ar[r]^{\simeq} & C(\xi)/I\ar[dd]_{\cong} \\
\ata\ar[ur]\ar[dr] && \\
&C(\xi') \ar[r]^{\simeq} & C(\xi')/I.}\]
Thereafter, 
taking a relative Sullivan model of $\ata \to C(\xi)$ and using the lifting lemma we obtain the following CDGA diagram 
\[\xymatrix{\ata \ar@{ >->}[d]\ar[rr] && C(\xi') \ar@{->>}[d]_{\simeq} \\
\relm\ar@{-->}[urr]^{\lambda} \ar[r]_-{\simeq} & C(\xi) \ar[r]_-{\simeq} & C(\xi)/I\cong C(\xi')/I.}\]
Hence, \[C(\xi)\simeq_{\ata} C(\xi').\]
\end{proof}
As a corollary we have:
\begin{corollary}
For $A$ a $1$-connected Poincar\'e duality CDGA of formal dimension $n$ odd, we have a natural surjection
\[\xymatrix{\psi \colon H^{2n-2}(\ata)/\left([\Delta]\right) \ar@{->>}[r]&\{C(\xi) \colon \xi \in (\ata)^{2n-2}\}/\simeq_{\ata},}\]
where the codomain of the map $\psi$ is the set of equivalence classes under $\ata$ of the CDGA $C(\xi)$.\end{corollary}
We conjecture the following:
\begin{conjecture}
\label{conj:xi}
The map $\psi$ of the previous corollary is a bijection.
\end{conjecture}
The interest of this conjecture is to link it to the next proposition and then obtain a more intrinsic caracterisation of the CDGA $C(\xi)$.
\begin{proposition}
\label{p:phi_iso}
For $A$ a $1$-connected Poincar\'e duality CDGA of formal dimension $n$ odd, we have a natural linear isomorphism 
\[\Phi \colon H^{n-2}(A) \stackrel{\cong}\longrightarrow \frac{H^{2n-2}(\ata)}{(\Delta)} \ ; \ [a] \mapsto [a\otimes \omega] \bmod ([\Delta]).\]
\end{proposition}
\begin{proof}

The morphism $\Phi$ is well defined since it is induced by the linear map, 
$A^{n-2} \to (A\otimes~A)^{2n-2}$ defined by $a \mapsto a\otimes \omega$,
 that sends coboundaries to coboundaries. To prove the surjectivity of $\Phi$, first of all, since $A$ is a 1-connected CDGA, we suppose, without loss of generality, that $A^0 \cong  \bq$, $A^1=0$ and $A^2 \cong H^2 (A)$. Henceforth, by Poincar\'e duality we have that $A^{n-2} \cong H^{n-2} (A)$. Then we notice that  
\[H^{2n-2}(\ata) \cong (\ata)^{2n-2} = \left(A^{n-2}\otimes \bq\cdot \omega\right) \oplus \left(\bq\cdot \omega \otimes A^{n-2}\right).\]
Thus, it suffices to prove that every element of the form $[\omega \otimes a] \bmod ([\Delta])$, with $a\in A^{n-2}$, is in the image of $\Phi$. A simple computation shows that 
$(a\otimes 1)\cdot \Delta= a\otimes \omega + \omega \otimes a,$
hence   $[\omega\otimes a]= \Phi([a]) \bmod ([\Delta]).$

To prove injectivity suppose that $\Phi([a])=0$, in other words $a\otimes \omega$ is cohomologous to a multiple of $\Delta$. The multiples of $\Delta$ in degree $2n-2$ are of the form $u\otimes \omega + \omega \otimes u$ with $u\in A^{n-2}$. For degree reasons and since $A^{n-1}=0$ (because $A$ is 1-connected), 
\[(\ata)^{2n-3}= ( \bq\cdot \omega \otimes A^{n-3})\oplus (A^{n-3}\otimes \bq\cdot\omega).\]
Therefore, there exists $x, y \in A^{n-3}$ such that
\[a\otimes \omega =  u\otimes \omega +\omega\otimes u + dx \otimes \omega +\omega \otimes dy. \]
Hence, $a=u+dx$ and $0=u+dy$, whence $a=d(x-y)$, namely $[a]=0$ in $H^{n-2}(A)$.
\end{proof}

Let $x\in H^{n-2}(M; \bq)$ and assume that $A$ is a 1-connected Poincar\'e duality CDGA model of $M$. We can choose $\xi \in (A\otimes A)^{n-2}$ such that $\Phi (x) = \xi \pmod \Delta$. Moreover, Proposition \ref{p:xi_xip} and Proposition \ref{p:phi_iso} show that the CDGA map 
\[\ata \to C(\xi)\] is independent of the choice of $\xi \in \ata$. We define then the CDGA \[C(x)\defeq C(\xi).\]
Theorem \ref{T:CDGA_FM2_odd} can be reformulated as follows:
\begin{corollary}
\label{C:CDGA_FM2_odd}
Let $M$ be a simply connected closed manifold of odd dimension $n$ and let $A$ be a $1$-connected Poincar\'e duality model of $M$ of formal dimension $n$. There exists $x \in H^{n-2}(M; \bq)$, such that the map 
\[\ata \to C(x)\]
is a CDGA model of the inclusion $F(M,2) \hookrightarrow M\times M$. 
\end{corollary}

We conclude with an example that illustrates the fact that not all the CDGA $C(\xi)$ are weakly equivalent under $\ata.$

\subsection* {The example $A= H^*(S^2\times S^3; \bq)$} 
\label{exemple}
Let us take the CDGA  \[(A,d)= \left(H^*(S^2\times S^3; \bq), 0\right)= \left(\{1,x_2, y_3, xy\}\cdot \mathbb{Q},0\right).\] It is easily verified that $A$ is a Poincar\'e duality CDGA and that the diagonal class in $\ata$ is given by
$\Delta = 1\otimes xy +x\otimes y - y\otimes x -xy\otimes 1.$
Let us define the morphism
$\Delta^! \colon \smn A \to A\otimes A \ ; \ \smn a \mapsto \Delta(1\otimes a).$
%
We will study the mapping cone  $(C(\Delta^!), \delta)$ 
and for the sake of conciseness we will use the notation $S=ss^{-5}.$

We can compute that 
$\delta (S1)= \Delta$, $\delta (Sx)=  x\otimes xy - xy \otimes x$, $\delta(Sy)= -y\otimes xy -xy\otimes y$,
$\delta(Sxy)=-xy\otimes xy.$

Let us quotient the cone $C(\Delta^!)$ by the differential acyclic submodule $\langle Sxy, xy\otimes xy\rangle$ and endow $C(\Delta^!)/\langle Sxy, xy\otimes xy\rangle$ with one of the CDGA structures described in Section \ref{s:odd}. For this, we have to define the product $S1\cdot S1$ which for degree reasons have to be an element in  $(\ata)^8$. For $q, r \in \mathbb{Q}$, set
\[ S1\cdot S1 = q (y\otimes xy) + r  (xy \otimes y)=\xi. \]
We denote by $C(q,r)=C(\xi)$ the obtained CDGA.

\begin{proposition}
\label{p:exemple}
Set $A= (H^*(S^2 \times S^3; \bq), 0)$ which is a Poincar\'e Duality CDGA model of $M=S^2\times S^3$ of formal dimesion $n=5$. Let $x, x' \in H^{n-2}(A)\cong \bq$, then 
\[C(x) \simeq_{\ata} C(x') \iff x=x'.\]
\end{proposition}

To prove Proposition \ref{p:exemple} let us construct the relative CDGA model of the morphism
\[\xymatrix{\ata \ar[r]^{id\oplus 0} & C(\Delta^!)\ar[r]^\pi & C(q,r)     }           \]
where $\pi: C(\Delta^!) \to C(q,r)$ is the canonical projection. 
We will construct a graded vector space $Z$, a differential $D$  and a CDGA quasi-isomorphism
$m: (\relmi_{q,r},D) \to C(q,r)$, such that the following diagram commutes 
\[\xymatrix{\ata\ar[dr] \ar[r]^{id\oplus 0} & C(\Delta^!)\ar[r]^\pi & C(q,r) \\
&(\relmi,D_{q,r}).\ar[ur]_{m}  &}\]

Constructing $Z$ degree by degree up to degree seven, we obtain the generators described below:
{\small
\begin{center}
\begin{tabular}{r|l|c| l}
Degree&$Z_0$ & $D$ & $m$ \\
\hline
4 &$u$ & $D(u)= \Delta:= 1\otimes xy + x\otimes y - y \otimes x -xy\otimes 1$ &$m(u)=s1$ \\
5&$z_5$ & $D(z_5)=u.(1\otimes x)-u.(x\otimes 1)$ & $m(z_5)=0$\\
6&$z_{61}$ & $D(z_{61})=u.(1\otimes y)-u.(y\otimes 1)$ & $m(z_{61})=0$ \\
6&$z_{62}$ & $D(z_{62})=z_5.(1\otimes x) + z_5.(x\otimes1 )$ & $m(z_{62})=0$ \\
7&$z_{71}$ & $D(z_{71})=z_{62} (1\otimes x)-z_{62} (x\otimes 1)$ & $m(z_{71})=0$\\
7&$z_{72}$ & $D(z_{72})=z_{61}(1\otimes x)+z_5(y\otimes1 )-z_5(1\otimes y)$ & $m(z_{72})=0$\\
&&\hspace{12mm}$-z_{61}(x\otimes1 )$&  \\
7&$h$      & $D(h)=u^2-2(z_{61}(1\otimes x)+z_{61}(x\otimes1 ))$& $m(h)=0 $\\
&& \hspace{12mm}$+q (y \otimes xy) + r (xy\otimes y)$ & 
\end{tabular}
\end{center}}

Since for all $C(q,r)$ the relative CDGA model is the same excepting the differential, we will denote by 
 $(\relmi, D_{q,r})$ the relative CDGA model associated to $C(q,r)$.

The following result is useful to determine when two of these algebras are weakly equivalent under $\ata$. The proof is left to the reader.
\begin{lemma}
\label{eqaa}
Let $A$ be a Poincar\'e duality CDGA of formal dimension $n$ odd and $\xi, \xi' \in (A\otimes\nolinebreak A)^{2n-2}$. Let 
 $\xymatrix{\ata\ar@{ >->}[r]& (\relmi_\xi, D_\xi)} \text{ and } \xymatrix{\ata\ar@{ >->}[r]& (\relmi_{\xi'}, D_{\xi'})}$ be relative CDGA models of $C(\xi)$ and $C(\xi')$ respectively.
We have that $C(\xi) \simeq_{\ata} C(\xi')$ if and only if there exists an isomorphism 
$\psi : (\relmi_\xi, D_\xi) \stackrel{\cong}{\longrightarrow} (\relmi_{\xi'}, D_{\xi'}),$ such that  $\psi (a\otimes b)= a\otimes b$, for all $a, b \in A$. 
\end{lemma}
%
%

Let us prove Proposition \ref{p:exemple}.
\begin{proof}[Proof of Proposition \ref{p:exemple}]
If $x=x'\in H^{n-2}(A)= H^3 (A)$ it is clear by Proposition \ref{p:xi_xip} and Proposition \ref{p:phi_iso}  that $C(x)\eqaa C(x')$. 

For the other implication, first notice that $H^3(A) =\langle [y] \rangle$. By the isomorphism $\Phi$ defined in Proposition \ref{p:phi_iso}, for $q\in \bq$, $q[y] \in H^3 (A)$ is sent to $[[qy\otimes xy]]$ in the quotient $H^{8}(\ata)/(\Delta)$.   A representative of the class $[[y\otimes xy]]\in H^{8}(\ata)/(\Delta)$ is given by $[y\otimes xy]$ in $H^8(\ata)$. Hence, by definition $C(q[y])= C(qy\otimes xy)=C(q,0)$. 

Let $r,q \in \bq$. If $C(q,0) \eqaa C(r,0)$, by Lemma \ref{eqaa} we have an isomorphism of CDGA \[\psi : (\relmi, D_{q,0}) \qi (\relmi, D_{r,0})\] such that $\psi (a\otimes b \otimes 1) = a\otimes b$.

For the sake of conciseness, when there is no ambiguity, both differentials ($D_{q,0}$ and $D_{r,0}$) will be denoted by $D$. For degree reasons  $\psi(u)= \alpha_1 u + \alpha_2 (x\otimes x) $
for some $\alpha_1 , \alpha_2 \in \mathbb{Q}$. We then have
$
D \psi (u) = D (\alpha_1 u + \alpha_2 (x\otimes x) )= \alpha_1 \Delta
$
and
$
\psi (Du) = \psi (\Delta)=\Delta.
$
Since $\psi$ commutes with differentials, we deduce that $\alpha_1=1$. Hence, 
$
\psi (u)= u + \alpha_2 (x\otimes x).
$
Analogously we have that since $z_{61}$ is of degree $6$, 
\begin{equation*}
\psi (z_{61}) = \beta z_{61}  + \beta_1 y\otimes y + \beta_2 u (1\otimes x) + \beta_3 u (x\otimes 1)+ \beta_4 z_{62},
\end{equation*}
for some $\beta_1$, $\beta_2$, $\beta_3$, $\beta_4$ and $\beta_5 \in \mathbb{Q}$. 
On one side
\begin{eqnarray*}
D\psi (z_{61})&=& D ( \beta_1 z_{61} + \beta_2 z_{62} + \beta_3 y\otimes y + \beta_4 u (1\otimes x) + \beta_5 u (x\otimes 1) )\\
&=& \beta u(1\otimes y - y\otimes 1) + (\beta_2 +\beta_3) (x\otimes xy - xy \otimes x) \\
&& +\beta_4 z_5 (1\otimes x + x\otimes 1),
\end{eqnarray*}
and on the other
\begin{eqnarray*}
\psi (Dz_{61}) &=& \psi (u (1\otimes y - y \otimes 1))\\
&=& u(1\otimes y - y\otimes 1) + \alpha_2 (x\otimes xy - xy \otimes x).
\end{eqnarray*}

Then, for $\psi$ to commute with the differentials, we have that $\beta=1$, $\beta_4=0$ and therefore 
\[
\psi (z_{61}) = z_{61} + \beta_1 (y\otimes y) + \beta_2 u (1\otimes x) + \beta_3 u (x\otimes 1)
\]
with $\beta_2 + \beta_3= \alpha_2$. 

Let's study what happens with the element $h\in Z^7$. For degree reasons
\begin{eqnarray*}
\psi(h)&=& \gamma_1 h + \gamma_2 (xy\otimes y) + \gamma_3(y \otimes xy) +\gamma_4 u (1\otimes y) + \gamma_5 u (y\otimes 1) \\
&& +\gamma_6 z_5 (1\otimes x) + \gamma_7 z_5 (x\otimes 1) +\gamma_8 z_{71} + \gamma_9 z_{72}.
\end{eqnarray*}
Simple calculations yield 
\begin{eqnarray*}
D_{q,0} \psi (h)
&=&  \gamma_1 u^2-2\gamma_1z_{61} (1\otimes x + x\otimes 1) + (\gamma_1q - (\gamma_4 +\gamma_5))(y\otimes xy) \\
&&  -(\gamma_4 + \gamma_5) (xy\otimes y)  +(\gamma_6 +\gamma_7) u (x\otimes x) \\
&& +\gamma_8 z_{62} (1\otimes x - x\otimes 1) \\
&&+ \gamma_9 (z_{61} (1\otimes x - x\otimes 1) + z_5 (y\otimes 1 - 1\otimes y))
\end{eqnarray*}
and on the other side
\begin{eqnarray*}
\psi (D_{r,0} (h) ) 
&=& u^2 + 2 (\alpha_2 - (\beta_2+\beta_3))u (x\otimes x) - 2z_{61} (1\otimes x + x\otimes 1)\\ && + (r-2\beta_1)(y\otimes xy) -\beta_1(xy\otimes y).
\end{eqnarray*}
For $\psi$ to commute with the differentials we must have that 
 $\gamma_1 =1$, $q - (\gamma_4 + \gamma_5) = r - 2\beta_1$ and $-(\gamma_4 + \gamma_5)= - 2\beta_1$. From these last equations we deduce that $q-r =(\gamma_4 +\gamma_5)-2\beta_1=0$, hence $q=r$.

\end{proof}

\begin{remark}
As a consequence of this last proposition we see that there are as many different (under $\ata$) CDGAs as elements in $H^3(A)$. An interesting fact to notice is that if we take out the condition "under $\ata$"  we obtain, using the same techniques as in the proof of Proposition \ref{p:exemple} (but this time without imposing that $\psi|_{\ata}= id$), \emph{only two} not weakly equivalent CDGAs: the CDGA $C(0)$ and $C(x)$ for $x\neq 0$. Indeed the CDGA $C(0)$ and the CDGA $C(x)$ are not weakly equivalent when $x\neq 0$. But, as CDGAs (not under $\ata$), $C(x)\simeq C(x')$ if $x\neq 0 \neq x'$.
\end{remark}

\section{Further questions}

The next step to understand the rational model of $F(M,2)$ for a simply connected manifold $M$ of odd dimension would be to prove Conjecture \ref{conj:xi}.  The fact of having a bijection between equivalence classes of CDGAs $C(\xi)$ and the cohomology group $H^{n-2}(M; \bq)$ would suggest (inspired by N. Habbeger's thesis \cite{HAB-PHD}) that the different CDGA structures $C(\xi)$ would be in one-to-one correspondance with the isotopy classes of embeddings  $M \hookrightarrow M\times M$ homotopic to the diagonal embedding.  The interest of this would be to find a geometrical argument that will give us information about the choice of $x\in H^{n-2}(M;\bq) $ such that $C(x)$ is a CDGA model of $F(M,2)$. 

An example of another geometrical argument that could give information about the choice of $x\in H^{n-2} (M;\bq)$ is 
%
 the fact that the inclusion 
\[F(M,2) \hookrightarrow M\times M\]
is $\Sigma_2$-equivariant (by the action of the symmetric group of two letters which acts by permutation of the coordinates). In \cite[section 6.4]{these} it is shown that this fact doesn't allow us to restrict the choice of $x \in H^{n-2}(M;\bq)$. 

To conclude, recall that in Section \ref{model_s2xs3}  we have introduced  the notion of an \emph{untwisted} manifold for which $C(0)$ is a model of the configuration space of two points. We also conjectured in that section that every simply connected closed manifold is untwisted. 
Oppositely, a very interesting fact will be to construct a \emph{twisted}  closed simply connected manifold $M$ of odd dimension (i.e. a manifold such that $C(0)$ is \emph{not} a rational model of $F(M,2)$).

\bibliographystyle{amsplain}
\bibliography{bibliography}

\providecommand{\bysame}{\leavevmode\hbox to3em{\hrulefill}\thinspace}
\providecommand{\MR}{\relax\ifhmode\unskip\space\fi MR }
\providecommand{\MRhref}[2]{%
  \href{http://www.ams.org/mathscinet-getitem?mr=#1}{#2}
}
\providecommand{\href}[2]{#2}
\begin{thebibliography}{1}

\bibitem{these}
Hector Cordova~Bulens, \emph{Mod{\`e}le rationnel du compl{\'e}mentaire d'un
  sous-poly{\`e}dre dans une vari{\'e}t{\'e} {\`a} bord et applications aux
  espaces des configurations}, Ph.D. thesis, Universit\'e catholique de
  Louvain, 2013.

\bibitem{FHT}
Yves F{\'e}lix, Stephen Halperin, and Jean-Claude Thomas, \emph{Rational
  homotopy theory}, vol. 205, Springer, 2001.

\bibitem{HAB-PHD}
Nathan Habegger, \emph{On the existence and classification of homotopy
  embeddings of a complex into a manifold}, Ph.D. thesis, Universit{\'e} de
  Gen{\`e}ve, 1981.

\bibitem{LSRHT}
Pascal Lambrechts and Don Stanley, \emph{The rational homotopy type of
  configuration spaces of two points}, Annales de l'institut Fourier, vol.~54,
  Chartres: L'Institut, 1950-, 2004, pp.~1029--1052.

\bibitem{LSPDCDGA}
\bysame, \emph{Poincar\'e duality and commutative differential graded
  algebras}, Ann. Sci. \'Ec. Norm. Sup\'er. (4) \textbf{41} (2008), no.~4,
  495--509. \MR{2489632 (2009k:55022)}

\bibitem{LSRDG}
\bysame, \emph{A remarkable {DG}module model for configuration spaces}, Algebr.
  Geom. Topol. \textbf{8} (2008), no.~2, 1191--1222.

\bibitem{LSFM2}
Riccardo Longoni and Paolo Salvatore, \emph{Configuration spaces are not
  homotopy invariant}, Topology \textbf{44} (2005), no.~2, 375--380.

\end{thebibliography}

\end{document}